\documentclass[preprint,12pt]{elsarticle}

\usepackage{amssymb}

\usepackage{lineno}

\journal{Information and Computation}

\begin{document}

\newcommand{\Z}{{\mathbb Z}}
\newcommand{\N}{{\mathbb N}}
\newcommand{\Hi}{{\mathbb H}}
\newcommand{\R}{{\mathbb R}}
\newcommand{\Q}{{\mathbb Q}}
\newcommand{\ICG}{\mathrm{ICG}}
\newcommand{\WCG}{\mathrm{WCG}}
\newcommand{\WICG}{\mathrm{WICG}}

\newtheorem{theorem}{\bf Theorem}[section]
\newtheorem{corollary}[theorem]{\bf Corollary}
\newtheorem{lemma}[theorem]{\bf Lemma}
\newtheorem{proposition}[theorem]{\bf Proposition}
\newtheorem{conjecture}[theorem]{\bf Conjecture}
\newtheorem{remark}[theorem]{\bf Remark}
\newtheorem{problem}[theorem]{\bf Problem}
\newtheorem{definition}[theorem]{\bf Definition}

\newcommand{\comment}[1]{\ $[\![${\normalsize #1}$]\!]$ \ }
\newcommand{\QED} {\hfill$\square$}

\newenvironment{proof} {\par \noindent \textbf{Proof. }}{\QED \par \bigskip \par}
\def\slika #1{\begin{center} \epsffile{#1} \end{center}}

\begin{frontmatter}



\title{Maximal diameter of integral circulant graphs}


\author[address1]{Milan Ba\v si\'c\corref{mycorrespondingauthor}}
\cortext[mycorrespondingauthor]{Corresponding author}
\ead{basic\_milan@yahoo.com}

\author[address2]{Aleksandar Ili\'c}
\ead{aleksandari@gmail.com}
\author[address1]{Aleksandar Stamenkovi\'c}
\ead{aca@pmf.ni.ac.rs}

\address[address1]{Faculty of Sciences and Mathematics, University of Ni\v{s}, Serbia}
\address[address2]{Facebook Inc, Menlo Park, California,USA}
\begin{abstract}
Integral circulant graphs are proposed  as
models for quantum spin networks that permit a quantum phenomenon called perfect state transfer.
Specifically, it is important to know how
far information can potentially be transferred between nodes of the
quantum networks modelled by integral circulant graphs and this task is related to calculating the maximal diameter of a graph.
The integral circulant
graph $\ICG_n (D)$ has the vertex set $Z_n = \{0, 1, 2, \ldots, n -
1\}$ and vertices $a$ and $b$ are adjacent if $\gcd(a-b,n)\in D$,
where $D \subseteq \{d : d \mid n,\ 1\leq d<n\}$.
Motivated by the result on the upper bound of the diameter of $\ICG_n(D)$ given in [N. Saxena, S. Severini, I. Shparlinski, \textit{Parameters of
integral circulant graphs and periodic quantum dynamics}, International Journal of Quantum Information 5 (2007), 417--430],
 according to which $2|D|+1$ represents one such bound,
 in this paper we prove that the maximal value of the diameter of the integral circulant graph $\ICG_n(D)$ of a given order $n$ with its prime factorization $p_1^{\alpha_1}\cdots p_k^{\alpha_k}$, is equal to $r(n)$ or $r(n)+1$, where  $r(n)=k + |\{ i \ | \alpha_i> 1,\ 1\leq i\leq k \}|$, depending on whether $n\not\in 4\N+2$ or not, respectively. Furthermore, we show that, for a given order $n$, a divisor set $D$ with $|D|\leq k$ can always be found such that this bound is attained.
Finally, we calculate the maximal diameter in the class of integral circulant graphs of a given order $n$
and cardinality of the divisor set $t\leq k$ and characterize all extremal graphs.
We actually show that the maximal diameter can have
the values $2t$, $2t+1$, $r(n)$ and $r(n)+1$ depending on the values of $t$ and $n$.
This way we further improve the upper bound of Saxena, Severini and Shparlinski and we also characterize all graphs whose diameters are equal to $2|D|+1$, thus generalizing a result in the above mentioned paper. 
\end{abstract}


\end{frontmatter}


\section{Introduction}
\label{S:1}

Circulant graphs are Cayley graphs over a cyclic group. A graph is
called integral if all the eigenvalues of its adjacency matrix are
integers. In other words, the corresponding adjacency matrix of a
circulant graph is the circulant matrix (a special kind of a Toeplitz
matrix where each row vector is rotated one element to the right
relative to the preceding row vector).  Integral graphs are
extensively studied in the literature and there has been a vast
research on some types of classes of graphs with integral spectrum.
The interest for circulant graphs in graph theory and applications
has grown during the last two decades. They appear in coding theory,
telecommunication network, VLSI design, parallel and distributed computing (see \cite{hwang03} and references therein).

Integral circulant graphs  have found important applications in molecular
chemistry for modeling energy-like quantities such as the heat of
formation of a hydrocarbon \cite{Ilic09,IlBa11,RaVe09,SaSa12}.
Recently there has been a vast research on the interconnection
schemes based on circulant topology -- circulant graphs represent an
important class of interconnection networks in parallel and
distributed computing (see \cite{hwang03}). Recursive circulants are proposed as an interconnection structure
for multicomputer networks \cite{PaCh00}.  While
retaining the attractive properties of hypercubes such as
node-symmetry, recursive structure, connectivity etc., these graphs
achieve noticeable improvements in diameter.

As circulants possess many
interesting properties (such as vertex transitivity called mirror
symmetry), they are applied in quantum information
transmission and proposed as models for quantum spin networks that
permit the quantum phenomenon called perfect state transfer
\cite{Ba10,Ba14,severini}. In the quantum communication scenario, the
important feature of this kind of quantum graphs (especially those
with integral spectrum) is the ability of faithfully transferring
quantum states without modifying the network topology. More precisely, these
networks consist of $n$ qubits where some pairs of qubits are
coupled via XY-interaction and are described by an undirected graph $G$ on a vertex set
$V(G)=\{1,2,\ldots,n\}$, where there is an edge between
vertices $i$ and $j$ if $i$-th and $j$-th qubit are coupled. Formally, to each vertex $i \in V$ , one can associate a
Hilbert space ${\cal H}_i\simeq C^2$, so that the Hilbert space associated with $G$ is given by
$ {\cal H}_G = \bigotimes_{i\in V(G)} {\cal H}_i =  (C^2)^{\otimes n}.$
 In \cite{fizicarski} a simple XY coupling is considered  such that the dynamics of the system is
governed by the XY Hamiltonian

$$ \label{eq:weighted Hamiltonian} H_G=\frac 1
2 \sum_{(i,j)\in
E} \sigma_i^x\sigma_j^x+\sigma_i^y\sigma_j^y,
$$
where the symbols $\sigma_i^x$ , $\sigma_i^y$ and $\sigma_i^z$
i denote the Pauli matrices acting on the Hilbert
space ${\cal H}_i$. Since the operator
of total $z$ component of the spin $\sigma^z_{tot}=\sum_{i=1}^n
\sigma_i^z$ commutes with $H_G$, the Hilbert space $\cal H_G$ is then decomposed into
invariant subspaces, each of which is a distinct eigenspace of the operator $\sigma^z_{tot}$.
For the purpose of quantum state transfer, it suffices to restrict our attention to
the single excitation subspace ${\cal S}_G$, which is the $n$-dimensional eigenspace of $\sigma^z_{tot}$,
corresponding to the eigenvalue $1 - n/2$ . The restriction of the Hamiltonian $H_G$ to the
subspace $\mathcal{S}_G$ is an $n \times n$ matrix identical to the
adjacency matrix $A_G$ of the graph $G$ and hence the time evolution operator can be written in the form
$F(t)=\exp(\i A_Gt)$. {\it Perfect state transfer} between different vertices
(qubits) $i$ and $j$ ($1 \leq i,j \leq n$) is obtained in time
$\tau$, if $|F(\tau)_{ij}|= e_i^T F(t)\ e_j =1$, where $e_i=(0,\ldots,0,1_i,0,\ldots,0)$ is the standard basis in ${\R}^n$ (if $|F(\tau)_{ii}|=1$, we say the graph is periodic at vertex $i$).
Using the spectral decomposition of $A_G$ we represent the perfect state transfer in terms its eigenvalues $\lambda_k$ and eigenvectors $u_k$, $1\leq k\leq n$, as
$ e_i^T (\sum_{k=1}^{n} e^{- i\lambda_k t}u_ku_k^T)\ e_j =1$.
In \cite{severini} was shown
that a quantum network whose hamiltonian is identical to the
adjacency matrix of a circulant graph is periodic if and only if all
eigenvalues of the graph are integers (that is, the graph is
integral). Therefore, circulant graphs having perfect state transfer must be
integral circulant graphs.

In this way we see that both combinatorial theory and the theory of graph spectra (spectral theory of its adjacency matrix) are relevant for the problems that we consider in
quantum computing, but also in the fields such as multiprocessor interconnection networks or complex networks.
Therefore, various spectral and combinatorial properties of integral circulant graphs were
recently investigated especially having in mind the study of certain parameters useful
for modeling  a good quantum (or in general complex) network that allows periodic dynamics.
 It was observed that integral circulant graphs
represent very reliable networks, meaning that the vertex
connectivity of these graphs  is equal to the degree of their
regularity. Moreover, for even orders they are bipartite \cite{severini} -- note that many of the
proposed networks mainly derived from the hypercube structure by
twisting some pairs of edges (twisted cube, crossed cub, multiply
twisted cube, M\"{o}bius cube, generalized twisted cube) are
nonbipartite. Other important network metrics of integral circulant graphs are analyzed as well, such as the degree, chromatic number, the
clique number, the size of the automorphism group, the size of the longest
cycle, the number of induced cycles of prescribed order (\cite{basic08,bail11,berrizbeitia04,fuchs05,Ilic10,klotz07}).

In this paper we continue the study of properties of circulant networks  relevant for
the purposes of information transfer.  Specifically, it would be interesting to know how
far information can potentially be transferred between nodes of the
networks modeled by the graph. So, it is important to know
the maximum length of all shortest paths between
any pair of nodes. Moreover, for a fixed number of nodes in the network, a larger diameter
potentially implies a larger maximum distance between nodes for which (perfect) transfer is possible (communication distance).
Therefore, for a given order of a circulant graph the basic question is to find the circulant graph which maximizes the diameter.
The sharp upper bound of the diameter of a graph of a given order is important for estimating the degradation
of performance of the network obtained by deleting a set of a certain number of vertices.


Throughout the paper we let $\ICG_n(D)$ be an arbitrary integral circulant graph of order $n$ and set of divisors $D$.
In \cite{severini}, the following sharp bounds on the diameter of integral circulant graphs are obtained

\begin{theorem}
Let $t$ be the size of the smallest set of additive generators
of $Z_n$ contained in $D$. Then
$$t \leq diam(\ICG_n(D)) \leq 2t + 1.$$ \QED
\end{theorem}

Direct consequences of the previous assertion are the following bounds
\begin{eqnarray}
\label{enq:main}
2 \leq diam(\ICG_n(D)) \leq 2|D| + 1.
\end{eqnarray}

\smallskip

One of the main results in this paper shows the maximal diameter of integral circulant
graphs of a given order $n$ to be equal to $r(n)+1$, if $n\not\in 4\N+2$, and $r(n)$, if $n\in 4\N+2$, where
$$
r (n) = k + |\{ i \ | \alpha_i> 1,\ 1\leq i\leq k \}|
$$
for $n=p_1^{\alpha_1}\cdots p_k^{\alpha_k}$. 
From these results we immediately conclude that the upper bound from (\ref{enq:main}) is never attainable for $|D|>k$
(as $|D|$ grows, the difference between the upper bound from (\ref{enq:main}) and the actual diameter of the graph grows as well).
Furthermore, since $diam(\ICG_n(D))\leq r(n)+1\leq 2k+1$ and the fact that the average order of $k$ is $\ln \ln n$ (\cite[p.~355]{HardyWright}) we conclude that $diam(\ICG_n(D))=O(\ln \ln n)$.

In Section 2 we prove that the bounds $r(n)$ and $r(n)+1$ are attainable for integral circulant graphs
$\ICG_n(D)$ such that $|D|\leq k$. This fact helps us to prove that the maximal diameter of all integral circulant graphs $\ICG_n(D)$
of a given order $n$ and the cardinality of the divisor set $D$ equal to $k$, is  $r(n)$ (Lemma 3.1).
This result directly implies that the diameters of graphs in this class of graphs are less or equal to $2k$, whence we conclude that the bound (\ref{enq:main}) is not attainable either.
Moreover, in Theorem 3.2 we find all $\ICG_n(D)$, $|D|=k$ whose diameter is equal to $r(n)$.
In Theorems 3.4 and 3.6 we consider integral circulant graphs $\ICG_n(D)$ such that there exists a prime $p_i$ not dividing any $d\in D$
and calculate that maximal diameter of this class of graphs is equal to $r(n)$.
After characterizing all extremal graphs, we also determine all $\ICG_n(D)$ such that there exists a prime $p_i$ not dividing any $d\in D$ whose diameter is equal to $2|D|+1$.
This class of graphs generalizes the example given in Theorem 5 from \cite{severini}.
Finally, using Theorem 3.8 we determine the maximal diameter of integral circulant graphs $\ICG_n(D)$ of a given order $n$ and cardinality $t\leq k$ of the divisor set, and characterize all extremal graphs.
We actually show that the maximal diameter can have
the values $2t$, $2t+1$, $r(n)$ and $r(n)+1$ depending on the values of $t$ and $n$, as given by the equation (\ref{eqn:main}).
The paper is concluded with a partial result about the maximal diameter of integral circulant graphs having perfect state transfer.
We also give some comments considering lower bounds of diameters of integral circulant graphs.

\section{Preliminaries}

In this section we introduce some basic notations and definitions. 

\smallskip

 A {\it circulant graph} $G(n;S)$ is a
graph on vertices $\Z_n=\{0,1,\ldots,n-1\}$ such that vertices $i$
and $j$ are adjacent if and only if $i-j \equiv s \pmod n$ for some
$s \in S$. The set $S$ is called the {\it symbol set} of graph $G(n;S)$.
As we will consider undirected graphs without loops, we assume that
$S=n-S=\{n-s\ |\ s\in S\}$ and $0\not\in S$. Note that the degree of
the graph $G(n;S)$ is $|S|$.

\smallskip

A graph is {\it integral} if all its eigenvalues are integers. A
circulant graph $G(n;S)$ is integral if and only if
$$
S=\bigcup_{d \in D} G_n(d),
$$
for some set of divisors $D \subseteq D_n$ \cite {wasin}. Here
$G_n(d)=\{ k \ : \ \gcd(k,n)=d, \ 1\leq k \leq n-1 \}$, and $D_n$ is
the set of all divisors of $n$, different from $n$.

Therefore an {\it integral circulant graph} (in further text $\ICG$)
$G(n;S)$ is
defined by its order $n$ and the set of divisors $D$. 
An integral circulant graph with $n$ vertices, defined by the set of
divisors $D \subseteq D_n$ will be denoted by $\ICG_n(D)$.
The term `integral circulant graph' was first introduced
in the work of So, where the characterization of the class of circulant graphs with integral spectra was given. The class of
integral circulant graphs is also known as `gcd-graphs' and arises as a generalization of unitary Cayley graphs \cite{basic08,klotz07}.

From the
above characterization of integral circulant graphs we have that the
degree of an integral circulant graph is $\deg \ICG_n(D)=\sum_{d \in
D}\varphi(n/d). $ Here $\varphi(n)$ denotes the Euler-phi function
\cite{HardyWright}. If $D=\{d_1,\ldots,d_t\}$, it can be seen
that $\ICG_n(D)$ is connected if and only if
$\gcd(d_1,\ldots,d_t)=1$, given that $G(n;s)$ is connected if and
only if $\gcd(n, S)=1$.



Recall that the greatest distance between any pair of vertices in a graph is graph diameter. The distance $d(u,v)$ between two vertices $u$ and $v$ of a graph is the minimum length of the paths connecting them.

Throughout this paper we let that the order of $\ICG_n(D)$ has the
following prime factorization $n=p_1^{\alpha_1}\cdots
p_k^{\alpha_k}$. Also, for a given prime number $p$ and an integer $n$,
denote by $S_p(n)$ the maximal number $\alpha$ such that $p^{\alpha}
\mid n$. If $S_p(n)=1$ we write $p\|n$.

\section{Main results}

Let $G = \ICG_n(D)$ be a connected graph with maximal diameter in the class of all integral circulant graphs of order $n$, and $D = \{d_1, d_2, \ldots, d_t\}$.

Let $D'$ be an arbitrary subset of $D$, such that $\gcd(\{d\ |\ d\in D'\})=1$. The graph $\ICG_n(D')$ is
connected and clearly a subgraph of $\ICG_n(D)$, so it follows that
$$
diam(\ICG_n(D))\leq diam(\ICG_n(D')).
$$

Since $\ICG_n(D)$ has maximal diameter, we have $diam(\ICG_n(D'))=$ \\$diam(\ICG_n(D))$. Therefore, we will find maximal diameter among all $\ICG_n(D)$
such that for every subset of
divisors $D'\subset D$ it holds that $\gcd(\{d\ |\ d\in D'\}) > 1$ (the graphs $\ICG_n(D')$ are unconnected). Furthermore, from the last assumption it follows that $gcd(d_1,\ldots , d_{s-1},
d_{s+1}, \ldots, d_t)> 1$ for every $1\leq s \leq t$ and  from the connectedness of $\ICG_n(D)$ we have $\gcd (d_1, d_2, \ldots, d_t) = 1$.
Thus, we conclude that  for each $s$ there exists a prime divisor $p_{i_s}$ of $n$ such that
$p_{i_s}\nmid d_s$ and $p_{i_s} \mid d_j$ for all $1\leq j \neq s\leq t$.

Therefore, we may define a bijective mapping
$$
 f : \{d_1, \ldots , d_t\} \rightarrow \{p_{i_1}, \ldots , p_{i_t}\},
$$ since $d_{s_1} \neq d_{s_2}$ implies $p_{i_{s_1}} \neq p_{i_{s_2}}$. Finally, we
conclude that for every divisor $d_s$, $1\leq s\leq t$, it holds that
\begin{eqnarray}
\label{property divisors}
p_{i_s} &\nmid& d_s \\
\label{property divisors1}
p_{i_1},\ldots,p_{i_{s-1}},p_{i_{s+1}}, \ldots, p_{i_t} &\mid& d_s.
\end{eqnarray}

So, in the rest of the paper we will assume that the divisors of the
integral circulant graph $\ICG_n(D)$ have the property described above,
unless it is stated otherwise.

\bigskip

In the following example we show how to narrow the search for the maximal diameter among all integral circulant graphs with order $n=12$ and divisor set $D'$
that has the properties (\ref{property divisors}) and (\ref{property divisors1}). It is easy to see that we have only two such examples $D'=\{2,3\}$ and $D'=\{3,4\}$ and the diameters of the graphs  $\ICG_{12}(\{2,3\})$ and  $\ICG_{12}(\{3,4\})$ are $2$ and $3$, respectively. However, the diameter of the graph $\ICG_{12}(\{3,4,6\})$,
as an example of a graph not having the properties (\ref{property divisors}) and (\ref{property divisors1}), does not exceed the diameters of the previous two graphs (Fig. 1).

\begin{figure}[h]
\label{Pic:ICGs}
\begin{center}
\includegraphics[width=5truecm]{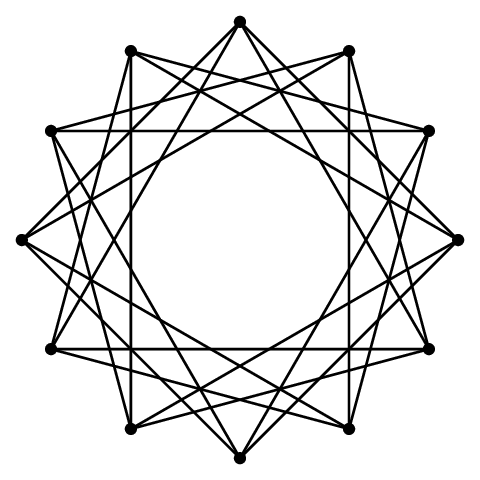} \hskip1truecm
\includegraphics[width=5truecm]{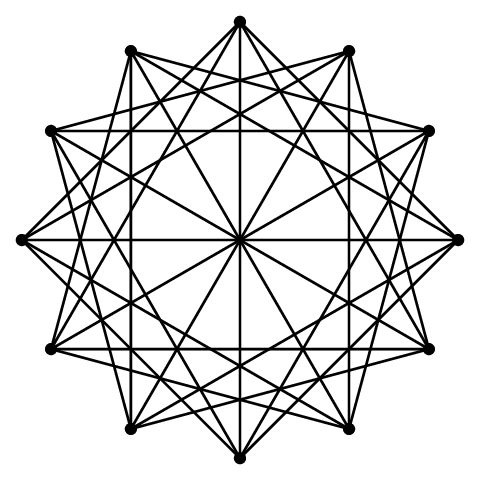}
\end{center}
\caption{Integral circulant graphs $\ICG_{12}(\{3,4\})$ on the left and $\ICG_{12}(\{3,4,6\})$ on the right whose diameters are equal to $3$ and $2$, respectively}
\end{figure}

\bigskip

Recall that two vertices $a$ and $b$ of $\ICG_n (d_1, d_2, \ldots, d_t)$ are
adjacent if and only if  $\gcd (a -  b,n) = d_i$ for some $1 \leq i \leq t$.
Given that the graph $\ICG_n(D)$ is vertex-transitive, we are going to focus on the vertex $0$ and
construct shortest paths from the vertex $0$ to any other vertex $0 \leq l=|a-b| \leq n - 1$.
\medskip

If we manage to find a solution $x = (x_1, x_2, \ldots, x_q)$ to
the following system of equations
$$
x_1 + x_2 + \cdots + x_q \equiv l \pmod n
$$
and
$$
\gcd(x_i, n) = d_{h(i)}
$$
for $1 \leq i \leq q$ and some function $h:\{1,\ldots
q\}\mapsto\{1,\ldots,t\}$, then we can construct the path from
$0$ to $l$ of  length $q$ passing through the vertices
$$
0, x_1, x_1 + x_2, \ldots, x_1 + x_2 + \cdots + x_q.
$$

We are going to look for  solutions of the above system of congruence equations and
greatest common divisor equalities, in the following representation
\begin{eqnarray}
\label{congruence n} d_1 y_1 + d_2 y_2 + \cdots + d_t y_t  \equiv l
\pmod n
\end{eqnarray}
with the constraints
\begin{eqnarray}
\label{congruence gcd} \gcd(d_j y_j, n) = d_j \qquad \mbox{for}
\qquad 1 \leq j \leq t,
\end{eqnarray}
where each summand $d_j y_j$ corresponds to one or more $x_i$'s. Notice that some of the $y_j$, $1\leq j\leq t$, can be equal to zero and in that case we do not  consider the constraint (\ref{congruence gcd}).

\bigskip

In the following lemma we prove that the upper bound of the diameter of $\ICG_n (d_1,d_2,\ldots, d_t)$, for $t=k$, is equal to $k + |\{ i \ | \alpha_i> 1,\ 1\leq i\leq k \}|$ and we denote this value by $r(n)$.

\begin{lemma}
\label{lem:auxiliary}
The diameter of the integral circulant graph $ \ICG_n (d_1,
d_2,\ldots, d_k)$, where the set of divisors $D=\{d_1, d_2,\ldots,
d_k\}$ satisfies the properties (\ref{property divisors}) and (\ref{property divisors1}), is less or
equal to $r (n)$.
\end{lemma}
\begin{proof}
Since the order of the divisors in $D$ is arbitrary and $D$ satisfies (\ref{property divisors}) and (\ref{property divisors1}), we may
suppose with out loss of generality that $i_s=s$, for $1\leq s\leq k$.
We are going to use the representation (\ref{congruence n}) for $t=k$, which requires solving the following system
\begin{equation}
\label{modulo pj^alphaj}
d_1 y_1 + d_2 y_2 + \cdots + d_k y_k \equiv l \pmod {p_{j}^{\alpha_{j}}}\quad 1\leq j\leq k,\\
\end{equation}
with the constraints
\begin{eqnarray*}
 \gcd(d_j y_j, n) = d_j \qquad \mbox{for} \qquad 1 \leq j \leq k.
\end{eqnarray*}
The last equation is equivalent to
\begin{eqnarray}
\label{congruence gcd1} \gcd( y_j, \frac{n}{d_j}) = 1\qquad \Leftrightarrow \qquad
\ p_i\nmid y_j,\ \mbox{if}\ S_{p_i}(d_j)<\alpha_i, 1\leq i\leq k.
\end{eqnarray}

From (\ref{property divisors}) we have $S_{p_j}(d_j)=0<\alpha_j$ and it directly holds that $p_{j}\nmid y_j$. 
\bigskip

Furthermore, as $\gcd(d_j, p_j^{\alpha_j}) = 1$ it follows
$$
y_j \equiv (l-p_jt) \cdot d_j^{-1} \pmod {p_j^{\alpha_j}},
$$
where by $p_jt$ we denote $d_1y_1+\cdots+d_{j-1}y_{j-1}+d_{j+1}y_{j+1}+\cdots+d_ky_k$.
Note that the only constraint is that $y_j$ cannot be divisible by $p_j$ as we already discuss above.  
\medskip

If $p_j$ does not divide $l$, we can directly compute $y_j$ which is not divisible by $p_j$. Assume now that $p_j$ divides $l$.

\medskip

For $\alpha_j = 1$, as $p_j\mid l$ we have
$$
y_j \equiv (l-p_jt) \cdot d_j^{-1} \equiv 0 \pmod {p_j^{\alpha_j}}
$$
and we can decide not to include the summand $d_j y_j$ in the above summation (otherwise, we obtain a contradiction with the constraint $p_j\nmid y_j$). This basically means that we are going to put $y_j = 0$ (ignore that summand) and in this case we do not consider the constraint $\gcd(d_jy_j,n)=d_j$.

\medskip

For $\alpha_j > 1$, assume that
\begin{equation}
\label{eq:less r(n)}
y_j \equiv (l-p_jt) \cdot d_j^{-1} \equiv p_j^{\beta_j} \cdot s \pmod {p_j^{\alpha_j}},
\end{equation}
where $p_j\nmid s$. For $\beta_j = \alpha_j$, we can similarly drop the summand $d_j y_j$ from the summation.

\medskip

If $\beta_j<\alpha_j$, then $y_j=0$ obviously is not a solution of the congruence equation. The trick now is to split $y_j$ into two summands $y_j' + y_j''$ which are both coprime with $p_j$ and the sum is equal to $p_j^{\beta_j} \cdot s$ modulo $p_j^{\alpha_j}$. This can be easily done by taking
$y_j' = 1$ and $y_j'' = p_j^{\beta_j} \cdot s - 1$. Therefore, we split the summand $d_j y_j$ into two summands that satisfy all the conditions
$$
d_j \cdot 1 + d_j \cdot (p_j^{\beta_i} \cdot s - 1).
$$
This means that for the prime factors with $\alpha_j > 1$ we need two edges in order to construct
a path from $0$ to $l$ when $ p_j^{\beta_j}\| l-p_jt $, $\beta_j<\alpha_j$.

Finally, we conclude that if we want to construct a path from $0$ to $l$, for an arbitrary $l$, we need at most one edge that corresponds the modulo $p_j$, if $\alpha_j=1$, and
at most two edges that correspond the modulo $p_j$, if $\alpha_j>1$, for every $1\leq j\leq k$. This means that $diam(\ICG_n(d_1,\ldots,d_k))\leq r(n)$.

\end{proof}

In the following theorem we show that the maximal diameter of \\$\ICG_n (d_1,d_2,\ldots, d_t)$, for $t=k$, is equal to $r(n)$ and characterize all extremal graphs.

\begin{theorem}
\label{thm:r(n)}
The maximal diameter of the integral circulant graph \\$ \ICG_n (d_1,
d_2,\ldots, d_k)$, where the set of divisors $D=\{d_1, d_2,\ldots,
d_k\}$ satisfies the properties (\ref{property divisors}) and (\ref{property divisors1}), is
equal to $r (n)$.
 The equality holds in two following cases
\begin{itemize}
\item[i)] if $\alpha_j>1$ then $S_{p_j}(d_i)>1$, for $1\leq i\neq j\leq k$ 
\item[ii)] if $n\in 4\N+2$ and  $d_1\in 2\N+1$ then there exists exactly one
$2\leq j\leq k$ such that $p_j\| d_1$, $S_{p_j}(d_i)>1$, for $2\leq i\leq k$  and other prime factors $p_l$ ($1\leq l\neq j \leq k$) satisfy $S_{p_l}(d_i)>1$, if $\alpha_l>1$, for $1\leq i\neq l\leq k$. 
\end{itemize}
\end{theorem}

\begin{proof}


Since the diameter of $ \ICG_n (d_1,
d_2,\ldots, d_k)$ is less or equal to $r(n)$, where the set of divisors $D=\{d_1, d_2,\ldots,
d_k\}$ satisfies the properties (\ref{property divisors}) and (\ref{property divisors1}),  according to Lemma \ref{lem:auxiliary}, in the first part of the proof we analyze when the graph has diameter less than $r(n)$.
From (\ref{eq:less r(n)}) we conclude that this is the case  when the condition $l-p_jt \equiv 0 \pmod {p_j^{\alpha_j}}$ is satisfied, for $\alpha_j>1$ and $p_j\mid l$.

Observe that the sum $p_jt$ can be
rewritten in the following form
\begin{eqnarray}
\label{auxiliary form}
p_j(t_1+\cdots+t_{u_j})+p_j^2(t_{u_j+1}+\cdots+t_{k-1}),
\end{eqnarray}
where $u_j=|\{d_i\ |\ i\in\{1,\ldots,j-1,j+1,\ldots,k\},\
p_j\| d_i, \ d_iy_i=p_jt_i\}|$ and $p_j\nmid t_1,\ldots,t_{u_j}$.
From (\ref{congruence gcd1}) we conclude that this form is indeed always possible since it holds that $p_j\nmid y_1,\ldots,y_{u_j}$
as $1=S_{p_j}(d_i)<\alpha_j$, $1\leq i\leq u_j$. Furthermore, since we
take into consideration the values $d_iy_i$ modulo $p_j^{\alpha_j}$,
we can assume, with out loss of generality,  that
$S_{p_j}(p_j^{2}t_i)<\alpha_j$ for $u_j+1\leq i\leq k$, where $d_iy_i=p_j^2t_i$.
In that case, according to (\ref{congruence gcd1}) it
must be that $p_j\nmid y_i$, for all $1\leq i\leq k$.

\smallskip

Assume first that $u_j\geq 2$. We see that
$l-p_jt \equiv 0 \pmod {p_j^{\alpha_j}}$  if and only if\\
\begin{eqnarray}
\label{eqn: mod l}
t_1+\cdots+t_{u_j}+p_j(t_{u_j+1}+\cdots+t_{k-1})\equiv \frac{l}{p_j} \pmod
{p_j^{\alpha_j-1}}.
\end{eqnarray}
 Furthermore, since $p_j\nmid t_1$ we see that $t_1$ must satisfy the following system
 $t_1\equiv \frac{l}{p_j}
-t_2-\cdots-t_{u_j}-p_j(t_{u_j+1}+\cdots+t_{k-1})\pmod
{p_j^{\alpha_j-1}}$ and $t_1\not\equiv p_js_1 \pmod
{p_j^{\alpha_j-1}}$, for some $0\leq s_1<p_j^{\alpha_j-2}$.
Again, since $p_j\nmid
t_2$ too, we obtain that $t_2\not\equiv
\{\frac{l}{p_j}-t_3-\cdots-t_{u_j}-p_j(t_{u_j+1}+\cdots+t_{k-1})-p_js_1,p_js_2\}\pmod
{p_j^{\alpha_j-1}}$, for some $0\leq s_2<p_j^{\alpha_j-2}$.
Therefore, if $p_j>2$ we can find the value $t_2$ modulo
$p_j^{\alpha_j-1}$ which is a solution of the system if we
take arbitrary values $t_3,\ldots,t_{k-1}$ modulo
$p_j^{\alpha_j-1}$ (the value $t_1$ modulo
$p_j^{\alpha_j-1}$ can be obviously computed thereafter using the values
$t_2,\ldots,t_{k-1}$). If $p_j=2$ and $\frac{l}{p_j}-t_3-\cdots-t_{u_j}\in 2\N+1$,
we can not find an odd $t_2$  which would be a solution of the system.
For such $l$, in the same way we can prove that the system
$d_1(y_1^{(1)}+y_1^{(2)})+d_2y_2+\cdots
+d_{j-1}y_{j-1}+d_{j+1}y_{j+1}+\cdots+d_ky_k\equiv l \pmod
{p_j^{\alpha_j}}$, $p_j\nmid y_1^{(1)},y_1^{(2)},y_2,\ldots,y_k$,
has a solution by reducing it to the form
$t^1_1+t^2_1+t_2+\cdots+t_{u_j}+p_j(t_{u_j+1}+\cdots+t_{k-1})\equiv \frac{l}{p_j}\pmod
{p_j^{\alpha_j-1}}$. So, in this case we need at most two edges for the
parts of the path from $0$ to $l$ corresponding to the moduli $p_1$ and $p_j$, since $l-p_jt\equiv 0\pmod {p_j^{\alpha_j}}$ and $y_j=0$.
Finally, we conclude that the diameter of this graph can not reach the value $r(n)$, since $r(p_1^{\alpha_1})+r(p_j^{\alpha_j})>2$.

\medskip

Now, let $u_j=1$. If $p_j>2$, similarly to the previous case by examining the system $d_1(y_1^{(1)}+y_1^{(2)})+d_2y_2+\cdots
+d_{j-1}y_{j-1}+d_{j+1}y_{j+1}+\cdots+d_ky_k\equiv l \pmod
{p_j^{\alpha_j}}$ (which corresponds to the equation
$t^1_1+t^2_1+p_j(t_{2}+\cdots+t_{k-1})\equiv \frac{l}{p_j}\pmod
{p_j^{\alpha_j-1}}$) and $p_j\nmid y_1^{(1)},y_1^{(2)},y_2,\ldots,y_k$, we see that
we can find $t^1_1,t^2_1,t_{2},\ldots,t_{k-1}$ such that $l-p_jt\equiv 0 \pmod {p_j^{\alpha_j}}$ and therefore we can set $y_j=0$.
Furthermore, as $p_1\mid d_2,\ldots,d_k$ it holds that $d_1(y_1^{(1)}+y_1^{(2)})\equiv l \pmod
{p_1}$ and if both $y_1^{(1)},y_1^{(2)}\neq 0$ then we conclude that for $p_1=2$ and $l\in 2\N+1$ the parity of left and right hand sides of the equation is violated
($p_1\nmid y_1^{(1)},y_1^{(2)}$). So, in this case we need one extra edge and
therefore at most three edges for the parts of the diameter path corresponding to moduli $p_1$ and $p_j$.
Finally, we conclude that maximal diameter of this graph can be attained if the value $r(p_1^{\alpha_1})+r(p_j^{\alpha_j})$ is equal to $3$ and this is the case if only if $\alpha_1=1$.
If $p_j=2$, we have already proved that for $\frac{l}{p_j}\in 2\N$ the system $t^1_1+t^2_1+p_j(t_{2}+\cdots+t_{k-1})\equiv \frac{l}{p_j}\pmod
{p_j^{\alpha_j-1}}$, $p_j\nmid t^1_1,t^2_1$ has a solution (case $u_j=2$) and therefore $l-p_jt\equiv 0\pmod {p_j^{\alpha_j}}$ ($y_j=0$).
Considering the modulus $p_1$, we can find $y_1^{(1)},y_1^{(2)}$ not divisible by $p_1$ and $y_1^{(1)}+y_1^{(2)}=(l-p_1t')d_1^{-1}$, where
$p_1t'=d_2y_2+\ldots +d_{j-1}y_{j-1}+d_{j+1}y_{j+1}+\cdots+d_ky_k$ as $p_1>2$.
In this case we need at most two edges for the parts of the path from $0$ to $l$ corresponding to the moduli $p_1$ and $p_j$ (the diameter of this graph can not reach the value $r(n)$).
If $\frac{l}{p_j}\in 2\N+1$, assume that $p_1\nmid l$, then   by examining the system $d_1y_1+d_2y_2+\cdots
+d_{j-1}y_{j-1}+d_{j+1}y_{j+1}+\cdots+d_ky_k\equiv l \pmod
{p_j^{\alpha_j}}$ (which corresponds to the equation
$t_1+p_j(t_{2}+\cdots+t_{k-1})\equiv \frac{l}{p_j}\pmod
{p_j^{\alpha_j-1}}$) and $p_j\nmid y_1,y_2,\ldots,y_k$, we see that $t_1\in 2\N+1$ satisfies the following system
 $t_1\equiv \frac{l}{p_j}-p_j(t_{u_j+1}+\cdots+t_{k-1})\pmod
{p_j^{\alpha_j-1}}$ and $t_1\not\equiv p_js \pmod
{p_j^{\alpha_j-1}}$, for some $0\leq s<p_j^{\alpha_j-2}$.
Moreover, $y_1\equiv (l-p_1t')d_1^{-1} \pmod{p_1^{\alpha_1}}$ is a solution of the congruence equation modulo $p_1$ and
in this case we need one edge for the parts of the path from $0$ to $l$ corresponding to the moduli $p_1$ and $p_j$ (the diameter of this graph can not reach the value $r(n)$).
Finally, if $\frac{l}{p_j}\in 2\N+1$ and $p_1\mid l$, the following system can be similarly solved $d_2 y_2 + \cdots +d_{j} (y_{j}^1+y_{j}^2)+\cdots + d_k y_k \equiv l \pmod {p_{j}^{\alpha_j}}$, $p_j\nmid y_{j}^1,y_{j}^2$.

We have proved that if $u_i=1$, $n\in 2\N$ and
$l,d_1\in2\N+1$, we can find a solution of the congruence system
$d_1(y_1^{(1)}+y_1^{(2)}+y_1^{(3)})+d_2y_2+\cdots
+d_{j-1}y_{j-1}+d_{j+1}y_{j+1}+\cdots+d_ky_k\equiv l \pmod
{p_i^{\alpha_j}}$, $p_i\nmid
y_1^{(1)},y_1^{(2)},y_1^{(3)},y_2,\ldots,y_k$.
Notice, that if $u_j=1$ and $p_j\| d_1$  we can find a solution of the congruence system
$d_1(y_1^{(1)}+y_1^{(2)}+y_1^{(3)})+d_2y_2+\cdots
+d_{j-1}y_{j-1}+d_{j+1}y_{j+1}+\cdots+d_ky_k\equiv l \pmod
{p_j^{\alpha_j}}$, $p_j\nmid
y_1^{(1)},y_1^{(2)},y_1^{(3)},y_2,\ldots,y_k$. This implies that we
can set $y_j=0$ and conclude that we need three edges for the
parts of the path corresponding to the moduli $p_1$, $p_i$ and $p_j$ and this path can not attain the value $r(n)$. The same
conclusion holds for more than two prime divisors.

\smallskip

From the  discussion above we conclude that
the diameter path can attain the value $r(n)$,
if $S_{p_j}(d_i)>1$ for $\alpha_j> 1$ and
$i\in\{1,\ldots,j-1,j+1,\ldots,k\}$ (the case $u_j=0$).
In yet another case
the value $r(n)$ of the diameter can be
attained if $S_{p_j}(d_1)=1<S_{p_j}(d_i)$, $2\leq i\leq k$, and $p_j,d_1\in 2\N+1$, for $n\in 4\N+2$ (the case $u_j=1$).
Now, we prove that there exists a vertex $l_0$ where the distance
from $0$ to $l_0$ is equal to $r(n)$ in both of the mentioned cases.

In the first case, we derive  $l_0$ from the system of
congruence equations composed by choosing exactly one of the following equation, for each $1\leq j\leq k$ (which exists due to the Chinese remainder theorem)
\begin{eqnarray*}
&&l_0 \equiv -1 \pmod {p_{j}} \qquad \mbox{ if } \alpha_{j} = 1\\
&&l_0 \equiv p_{j} \pmod {p_{j}^{\alpha_{j}}} \qquad \mbox{ if }
\alpha_{j}> 1.
\end{eqnarray*}

For all $1\leq j\leq k$ such that $\alpha_{j} = 1$, we need at least
one summand $d_jy_j$ in the representation (\ref{modulo pj^alphaj}) of $l_0$
 since $l_0\not\equiv
0\pmod{p_{j}^{\alpha_{j}}}$ and all other $d_i$ are divisible by
$p_{j}^{\alpha_{j}}$ for $i \neq j$.  On the other hand, for all
$1\leq j\leq k$ such that $\alpha_{j} > 1$ we cannot have exactly
one such summand as otherwise we would have $d_jy_j\equiv
l_0-p_j^2t'\equiv p_{j}s\pmod{p_{j}^{\alpha_{j}}}$ and $p_j\nmid s$, which would be a
contradiction as $p_j\nmid d_jy_j$, where $p_j^2t'=d_1y_1+\cdots+d_{j-1}y_{j-1}+d_{j+1}y_{j+1}+\cdots+d_ky_k$.
Therefore, we need at least two summands for $\alpha_{j}> 1$.

\smallskip

In the second case, we derive $l_0$ from the following system of
congruence equations
\begin{eqnarray*}
&&l_0 \equiv -1 \pmod {p_{i}} \qquad \mbox{ if } \alpha_{i} = 1\\
&&l_0 \equiv p_{j}^{2} \pmod {p_{j}^{\alpha_{j}}} \qquad
\mbox{ if } S_{p_j}(d_1)=1<S_{p_j}(d_l),\ 2\leq l\leq k\\
&&l_0 \equiv p_{i} \pmod {p_{i}^{\alpha_{i}}} \qquad \mbox{ if }
\alpha_{i}> 1,\ i\neq j.\\
\end{eqnarray*}

It remains to prove that we need two edges by considering the modulus
$p_{j}^{\alpha_{j}}$, for $p_j, d_1\in 2\N+1$ and $S_{p_j}(d_1)=1<S_{p_j}(d_l),\ 2\leq l\leq k$.
 Indeed, we cannot have exactly one such
summand as otherwise we would have $d_jy_j\equiv
l_0-p_jt\equiv p_{j}s\pmod{p_{j}^{\alpha_{j}}}$ and $p_j\nmid s$,
which is a contradiction due to $p_j\nmid d_jy_j$.

This proves that the diameter of $\ICG_n(D)$ in these cases is
greater or equal to $r(n)$ and therefore $diam(\ICG_n(D))=r(n)$, which completes the proof of the
theorem.
\end{proof}

We provide several examples to illustrate how Theorem \ref{thm:r(n)} is applied to determine whether the diameter of $\ICG_n(D)$ such that $|D|=k$ attains $r(n)$.
\begin{itemize}
    \item [\bf $i)$] If $n=2^2\cdot 3^3\cdot 5$ and $D=\{3^2\cdot 5,2^2\cdot 5,2^2\cdot 3^3\}$ then $diam(\ICG_n(D))=5$, according to the first part of Theorem \ref{thm:r(n)} and therefore the diameter of the graph attains $r(n)$.

        We will also compute the diameter of $\ICG_n(d_1,d_2,d_3)$ for $n=2^2\cdot 3^3\cdot 5$,
        $d_1=3^2\cdot 5$, $d_2=2^2\cdot 5$ and $d_3=2^2\cdot 3^3$ to illustrate the methods from the proof of Theorem \ref{thm:r(n)}. Indeed, for any $0\leq l\leq n-1$ we will try to solve the equation in the following form $d_1y_1+d_2y_2+d_3y_3\equiv l\pmod n$, with the constraints $\gcd(d_jy_j,n)=d_j$, for $1\leq j\leq 3$. This is equivalent with the following system
        \begin{eqnarray*}
        y_1&\equiv& (l-d_2y_2-d_3y_3)\cdot d_1^{-1}\pmod {2^2}\\
        y_2&\equiv& (l-d_1y_1-d_3y_3)\cdot d_2^{-1}\pmod {3^3}\\
        y_3&\equiv& (l-d_1y_1-d_2y_2)\cdot d_3^{-1}\pmod 5,\\
        \end{eqnarray*}
        where $2\nmid y_1$, $3\nmid y_2$ and $5\nmid y_3$. If $2\nmid l$ then we can directly compute $y_1$. Similarly if $3\nmid l$ or $5\nmid l$, we can directly compute
        $y_2$ or $y_3$, respectively. In particular, if $2\nmid l$, $3\nmid l$ and $5\nmid l$ we conclude that the distance between $0$ and $l$ is equal to $3$.

        If $2^2\mid l-d_2y_2-d_3y_3$ or  $3^3\mid l-d_1y_1-d_3y_3$ or $5\mid l-d_1y_1-d_2y_2$ then we obtain that a solution of the corresponding equation is $y_1=0$ or $y_2=0$ or $y_3=0$, respectively.
        On the other hand, if $3^2\|l-d_1y_1-d_3y_3$ or $3\|l-d_1y_1-d_3y_3$ we can not find $y_2$ satisfying the above congruence equation such that $3\nmid y_2$.
        However, for such $l$ there exist $y_2'$ and $y_2''$ such that  $y_2'+y_2''\equiv (l-d_1y_1-d_3y_3)\cdot d_2^{-1}\pmod {3^3}$ and $3\nmid y_2',y_2''$. This means that we need two edges in order to construct
        a path from $0$ to $l$ when $3^2\|l-d_1y_1-d_3y_3$ or $3\|l-d_1y_1-d_3y_3$ corresponding to the modulus $3^3$. Similarly, if we take
        $2\|l-d_2y_2-d_3y_3$ in order to construct a path from $0$ to $l$ we need two edges corresponding to the modulus $2^2$. Finally, the distance between any two vertices is less or equal to $5$.
        According to the above discussion we can choose $l$ such that
        \begin{eqnarray*}
        l&\equiv& -1\pmod 5\\
        l&\equiv& 3\pmod {3^3}\\
        l&\equiv& 2\pmod {2^2},\\
        \end{eqnarray*}
        and calculate that $l=354$ using the Chinese Remainder Theorem. Therefore, we
         conclude that the distance between $0$ and $l$ is equal to $5$.

\item [\bf $ii)$] If $n=2\cdot 3^3\cdot 5^3$ and $D=\{3\cdot 5^2,2\cdot 5^3,2\cdot 3^2\}$ then  $diam(\ICG_n(D))=5$, due to the part $ii)$ of Theorem \ref{thm:r(n)} as there exists exactly one $p_j=3\mid d_1$ such that $S_{p_j}(d_1)=1<S_{p_j}(d_3)$ and $d_1\in 2\N+1$, for $n\in4\N+2$ and the diameter of the graph is equal to $r(n)=5$.

\item [\bf $iii)$] If $n=2^2\cdot 3^2\cdot 5\cdot 7$ and $D=\{3\cdot 5\cdot 7,2^2\cdot 5\cdot 7,2^2\cdot 3^2\cdot 7,2^2\cdot 3^2\cdot 5\}$
then $diam(\ICG_n(D))=5$ (less than $r(n)$), as $n\in 4\N$ even though the other conditions from the part $ii)$ of Theorem \ref{thm:r(n)} are satisfied.
On the other hand, since $S_{p_2}(n)>1$, for $p_2=3$, and $S_{p_2}(d_1)=1$, we conclude that first part of the theorem is not satisfied.

\item [\bf $iv)$] If $n=2^2\cdot 3\cdot 5\cdot 7$ and $D=\{3\cdot 5\cdot 7,2\cdot 5\cdot 7,2^2\cdot 3\cdot 7,2^2\cdot 3\cdot 5\}$ then $diam(\ICG_n(D))=4$ (less than $r(n)$), as all conditions from the part $i)$ of Theorem \ref{thm:r(n)} are satisfied except $S_{p_1}(d_2)>1$ for $p_1=2$ and $d_2=2\cdot 5\cdot 7$.

\item [\bf $v)$] If $n=2\cdot 3^2\cdot 5^2\cdot 7^2$ and $D=\{3\cdot 5\cdot 7,2\cdot 5^2\cdot 7^2,2\cdot 3^2\cdot 7^2,2\cdot 3^2\cdot 5^2\}$ then  $diam(\ICG_n(D))=5$ (less than $r(n)$), as the condition from the part $ii)$ of Theorem \ref{thm:r(n)} is satisfied for more than one divisor (in this for three divisors $p_j\in \{3,5,7\}$).
\end{itemize}

\bigskip

We also directly conclude that in the class of all integral circulant graphs with a given order $n$ and $k-$element divisor set $D$
we can not find an example for which $diam (\ICG_n(d_1,\ldots,d_k))=2k+1$, since $diam (\ICG_n(d_1,\ldots,d_k))\leq r(n)\leq 2k$.
So, in this case the upper bound from (\ref{enq:main}) is not attainable.

\bigskip

\begin{lemma}
\label{lem:2 and 3 summands}
Let $d$ be an arbitrary divisor of $n$. For any $l$ divisible by $d$ there exist $y_1$ and $y_2$
such that
\begin{eqnarray*}
&&d(y_1+y_2)\equiv l \pmod n, \mbox{ if } \frac n d \in 2\N+1 \\
&&d(y_1+y_2+1)\equiv l \pmod n, \mbox{ if } \frac n d \in 2\N
\end{eqnarray*}
and $\gcd(dy_i,n)=d$, $1\leq i\leq 2$.
\end{lemma}
\begin{proof}
 We find a representation of $l$
in the following form
$dy_1+dy_2\equiv l \pmod{n}$ such that
$\gcd(dy_1,n)=d$, $\gcd(dy_2,n)=d$ and prove that the representation of $l$ exists if $\frac{n}{d}\in 2\N+1$.
We actually need to find $y_1$ and $y_2$
such that $y_1+y_2\equiv \frac{l}{d}\pmod
{\frac{n}{d}}$, $\gcd(y_1,\frac{n}{d})=1$ and
$\gcd(y_2,\frac{n}{d})=1$.
Now, let $p_i$ be an arbitrary divisor of $\frac n {d}$ such that
$S_{p_i}(\frac{n}{d})=\alpha_i$.
By solving the above congruence
equation system modulo $p_i^{\alpha_i}$ we get that $y_1\equiv
\frac{l}{d}-y_2\not\equiv p_iu\pmod {p_{i}^{\alpha_i}}$ and
$y_2\not\equiv 0\pmod {p_{i}}$ (which is equivalent to
$y_2\not\equiv p_iv\pmod {p_{i}^{\alpha_i}}$), for $0\leq
u,v<p_i^{\alpha_i-1}-1$.
 Finally, we obtain $y_2\not\equiv
\{p_iv,\frac{l}{d}-p_iu\}\pmod {p_{i}^{\alpha_i}}$ and it can be
concluded that  the maximal number of values that $y_2$ modulo $p_i^{\alpha_i}$
can not take is equal to $2p_i^{\alpha_i-1}$. This
number is  less than the number of residues modulo $p_i^{\alpha_i}$
for $p_i>2$, so this
system has a solution in this case. Now, suppose that $\frac{n}{d}\in 2\N$.
In the
case for $p_i=2$ we observe that there also exists $y_2$ modulo $p_{i}^{\alpha_i}$ if $\frac{l}{d}\in 2\N$.
Furthermore, for $p_i=2$ and $\frac{l}{d}\in 2\N+1$, we have already concluded that there exists $y_1,y_2\in 2\N+1$ such that
$dy_1+dy_2\equiv l-d \pmod{p_i^{\alpha_i}}$ (notice that $\frac{l-d}{d}\in 2\N$) and therefore the representation
$dy_1+dy_2+d\equiv l \pmod{n}$ exists, under the constraints
$\gcd(dy_1,n)=d$ and $\gcd(dy_2,n)=d$.

\end{proof}

\begin{theorem}
\label{thm:d,d+1}
Let  $diam(\ICG_m(d_1,\ldots,d_t))=d>2$. For $n'>1$ such that
$\gcd(n',m)=1$ it holds that
\begin{eqnarray}
\label{eq:d,d+1}
diam(\ICG_{mn'}(d_1,\ldots,d_t))&=&\left\{ \begin{array}{rl}
d+1, & n' \in 2\N \\
d, & n' \in 2\N+1. \\
\end{array}\right.
\end{eqnarray}
\end{theorem}
\begin{proof}
Let $d_1'y_1+d_2'y_2+\cdots +d_s'y_s$ be a representation of $0\leq l\leq n'm-1$
modulo $m$ that satisfies the equations $\gcd(d_j'y_1,m)=d_j'$, such
that $0\leq s\leq d$ and $d_j'\in
\{d_1,\ldots,d_t\}$, $1\leq j\leq s$.

Now, we examine which additional conditions $y_1, y_2,\ldots,y_s$ must satisfy
by considering the following system
\begin{eqnarray*}
&&d_1' y_1 + d_2' y_2 +\cdots + d_s' y_s  \equiv l \pmod {mn'}\\
&&\gcd(d_j'y_j,mn')=d_j',\ 1\leq j\leq s.
\end{eqnarray*}

Since $\gcd(n',m)=1$, $\gcd(d_j'y_j,m)=d_j'$ and
$d_1'y_1+d_2'y_2+\cdots +d_s'y_s\equiv l \pmod {m}$, we conclude
that it is sufficient to find $y_1, y_2,\ldots,y_s$ such that for every $p_i\mid n'$ it holds

\begin{eqnarray}
&&d_1' y_1 + d_2' y_2 +\cdots + d_s' y_s  \equiv l \pmod {n'}\label{cond1:n'}\\
&&p_i\nmid y_j,\ 1\leq j\leq s,\ 1\leq i\leq u, \label{cond2:n'}
\end{eqnarray}
where $u$ denotes the number of prime factors of $n'$.

\medskip

Let $s=1$. From (\ref{cond1:n'}) and (\ref{cond2:n'}) follow that $y_1 \equiv d_1^{'-1}l\pmod {p_i^{\alpha_i}}$ and $y_1 \not\equiv 0 \pmod {p_i}$ must be satisfied,
for every $1\leq i\leq u$ .
But, if $p_i\mid l$ then  the system has no solution and we first proceed by considering the representation of $l$
in the following form
$d_1'y_1^{(1)}+d_1'y_1^{(2)}\equiv l \pmod{m}$ such that
$\gcd(d_1'y_1^{(1)},m)=d_1'$ and $\gcd(d_1'y_1^{(2)},m)=d_1'$. According to Lemma \ref{lem:2 and 3 summands} there exists such a representation if $\frac m {d_1'}\in 2\N+1$.
Moreover, using the same assertion there is a representation of $l$
in the following form
$d_1'y_1^{(1)}+d_1'y_1^{(2)}+d_1'\equiv l \pmod{m}$ such that
$\gcd(d_1'y_1^{(1)},m)=d_1'$ and $\gcd(d_1'y_1^{(2)},m)=d_1'$  if $\frac{m}{d_1'}\in 2\N$.

\medskip

Notice that $s=0$ if and only if $l$ is divisible by $m$. This means that $d_j\mid l$ for $1\leq j\leq t$ and hence
we can proceed as in the previous case by finding a representation of $l$ modulo $m$ on two or three summands depending on the parity of $m$.


\bigskip

Now, in the following we show that for every representation $d_1'y_1+d_2'y_2+\cdots +d_s'y_s\equiv l \pmod m$
with constraints $\gcd(d_j'y_1,m)=d_j'$ and $s>1$ there exists the representation (\ref{cond1:n'})
under the constraints (\ref{cond2:n'}) with $s$ or $s+1$ summands (depending on the parity of $n'$). Indeed, for any prime $p_i\mid n'$ ($1\leq
i\leq u$) where $S_{p_i}(n')=\alpha_i$, we have that
$\gcd(m,p_i)=1$ implying that $\gcd(d_j,p_i)=1$, for $1\leq j\leq
t$. Thus, from (\ref{cond1:n'}) and (\ref{cond2:n'}), we conclude
that
\begin{eqnarray*}
y_1&\equiv& (l-d_2'y_2-\cdots -d_s'y_s)d_1^{'-1}\pmod {p_i^{\alpha_i}}\\
y_1&\not\equiv& p_iv \pmod {p_i^{\alpha_i}},\ 0\leq
v<p_i^{\alpha_i-1}.
\end{eqnarray*}
From these two conditions it follows that
$(l-d_2'y_2-\cdots-d_s'y_s)d_1^{'-1}\not\equiv
p_iv\pmod{p_i^{\alpha_i}}$, $0\leq v<p_i^{\alpha_i-1}$, whence we
further obtain that
\begin{eqnarray*}
y_2&\not\equiv& (l-d_3'y_3-\cdots-d_s'y_s-p_id'_1v)d_2^{'-1}\pmod {p_i^{\alpha_i}}\\
y_2&\not\equiv& p_iw \pmod {p_i^{\alpha_i}},\ 0\leq
w<p_i^{\alpha_i-1}.
\end{eqnarray*}

Since $v$ and $w$ each may take $p_i^{\alpha_i-1}$ values, we conclude that the
maximal number of values modulo $p_i^{\alpha_i}$ that $y_2$ can not
take is equal to $2p_i^{\alpha_i-1}$. This number is obviously less
than the number of residues modulo $p_i^{\alpha_i}$ for $p_i>2$,
from where we further deduce that if $n'\in 2\N+1$, this system of
congruence equations has a solution modulo $n'$.

If $n'\in 2\N$ and we take $p_i=2$ then (\ref{cond2:n'}) implies that $y_1,\ldots
,y_s\in 2\N+1$ and as $d_1',\ldots ,d_s'$ are divisors of $m\in2\N+1$ it also holds that $d_1',\ldots ,d_s'\in 2\N+1$. On the other hand, if we choose
$l\not\equiv s \pmod 2$, then we obtain that $d_1'y_1+d_2'y_2+\cdots
+d_s'y_s\equiv s\not\equiv l\pmod 2$, which is a contradiction. This
means that we can not make the representation of $l$ into $s>0$
summands that satisfy the system of congruence equations
(\ref{cond1:n'}) and (\ref{cond2:n'}) for $n'\in 2\N$ and $l$
such that $l\not\equiv s \pmod 2$. However, since $l\not\equiv l-d_1'\pmod 2$, there is a
representation on $s$ summands of $l-d'_1$ and hence there is a representation on $s+1$ summands of $l$ such that $l\not\equiv
s \pmod 2$.

\smallskip

From the above discussion we conclude that for $n'\in 2\N+1$ ($m\in 2\N$)
the diameter of $\ICG_{mn'}(d_1,\ldots,d_t)$ is equal to $\max\{3,d\}=d$. On the other hand,
if $n'\in 2\N$ ($m\in 2\N+1$)
the diameter of $\ICG_{mn'}(d_1,\ldots,d_t)$ is equal to $\max\{2,d+1\}=d+1$.

\end{proof}

\begin{remark}
\label{rem:1}
Notice that for the complete graph $\ICG_m(D_m)$ the diameter of $\ICG_{mn'}(D_m)$ is equal to $2$ if $n'\in 2\N+1$ and equal to $3$ if $n'\in 2\N$.
Indeed, since for every edge $uv$ there exists the path $u-w-v$ we conclude that for every $0\leq l\leq m-1$ there exists the representation (\ref{congruence n}) modulo $n$ on $s =2$ summands and from the second part of the above proof the assertion directly holds.
Moreover, for the integral circulant graph such that $diam(\ICG_{m}(D))=2$ and $m\in 2\N+1$, according to the above proof we have that $diam(\ICG_{mn'}(D))=3$.
On the other hand, if $m\in 2\N$ and for every $0\leq l\leq m-1$ there exists a path $0-l_1-l$ for some $0\leq l_1\leq l\leq m-1$ the diameter of $\ICG_{mn'}(D)$ remains the same
as the diameter of $\ICG_{m}(D)$. Otherwise, $diam(\ICG_{mn'}(D))=3$.
\end{remark}

\bigskip

From Theorem \ref{thm:d,d+1} it follows that
$diam(\ICG_{n\cdot n'}(d_1,\ldots,d_k))\leq$\\ $ diam(\ICG_n(d_1,\ldots, d_k))+1$, for some $n'>1$ and $\gcd(n,n')=1$. Furthermore, by Theorem
\ref{thm:r(n)} we have $diam(\ICG_n(d_1,\ldots,d_k))\leq r(n)$, which implies
that $diam(\ICG_{n\cdot n'}(d_1,\ldots,d_k))\leq r(n)+1$. Finally, from the definition of the
function $r$, since $\gcd(n,n')=1$ and $n'>1$, follows that $r(n)+1\leq r(n \cdot n')$ and therefore $diam(\ICG_{n\cdot n'}(d_1,\ldots,d_k))\leq r(n\cdot n')$ holds.
Therefore, we can not increase the upper bound given in Theorem
\ref{thm:r(n)} of the diameter of integral circulant graphs using  Theorem \ref{thm:d,d+1}, but we can find new classes of graphs
that attain the bound.

We further see that the equation $diam(\ICG_{n\cdot n'}(d_1,\ldots,d_k))=r(n \cdot n')$ holds if and only if
$diam(\ICG_{n\cdot n'}(d_1,\ldots,d_k))= diam(\ICG_n(d_1,\ldots, d_k))+1$, $diam(\ICG_n(d_1,\ldots, d_k))=r(n)$ and $r(n)+1=r(n \cdot n')=r(n)+r(n')$.
By Theorem \ref{thm:d,d+1} we obtain that $diam(\ICG_{n\cdot n'}(d_1,\ldots,d_k))= diam(\ICG_n(d_1,\ldots, d_k))+1$ if and only if $n'\in 2\N$ and $n\in 2\N+1$.
Furthermore, since $n\in 2\N+1$, the diameter of $\ICG_n(d_1,\ldots, d_k)$ can attain the value
$r(n)$ only in the case given by the first part of the assertion of
Theorem \ref{thm:r(n)}. Finally, $r(n)+1=r(n)+r(n')$ holds if and only if $n'$ is prime and this proves the following theorem.

\begin{theorem}
\label{thm:main}
Let  $diam(\ICG_n(d_1,\ldots,d_k))>2$. The diameter of the integral circulant graph $ \ICG_{n\cdot n'} (d_1,
d_2,\ldots, d_k)$,  where $n'>1$, $\gcd(n,n')=1$ and the set of the divisors $\{d_1, d_2,\ldots,
d_k\}$ satisfies the properties (\ref{property divisors}) and (\ref{property divisors1}) with respect to primes $\{p_1,\ldots p_k\}$, is less or
equal to $r(n\cdot n')$. The equality holds if and only if
$n'=2$, $n\in 2\N+1$ and the sets of divisors $D\subseteq 2\N+1$ satisfies
the conditions $i)$ in the assertion of Theorem \ref{thm:r(n)}.
\end{theorem}

We have already mentioned that Saxena, et al. determine the upper bound of the diameter of
$\ICG_n(d_1,\ldots,d_t)$ in terms of the number of divisors $t$,
i.e. $diam(\ICG_n(d_1,\ldots,d_t))\leq 2t+1$. They also present a class of
integral circulant graphs in Theorem 5 \cite{severini}, where $n=2m$, $m=p_1^2\cdots p_k^2$ and $D_0=\{m/p_1^2,\ldots,m/p_k^2\}$  as an example for which the bound is tight.
However, using a similar approach as in the proof of Theorem \ref{thm:main} we can characterize all the graphs $\ICG_{n\cdot n'}(d_1,\ldots,d_k)$ with the diameters equal to $2k+1$.
First, notice that
$$
diam(\ICG_{n\cdot n'}(d_1,\ldots,d_k))\leq diam(\ICG_n(d_1,\ldots, d_k))+1\leq r(n)+1\leq 2k+1.
$$
The equality holds if and only if $n'\in 2\N$, $\alpha_1>1,\ldots,\alpha_k>1$ and and the set of divisors $\{d_1,\ldots,d_k\}\subseteq 2\N+1$ satisfies
the conditions $i)$ in the assertion of Theorem \ref{thm:r(n)}.
Finally, the diameter of $\ICG_{n\cdot n'}(d_1,\ldots,d_k)$, for $\gcd(n,n')=1$, attains the upper bound $2k+1$ if and only if $n'\in 2\N$ and for every $d_j\in 2\N+1$, $1\leq j\leq k$, $p_{j}\nmid d_j$, $p_{i}^2\mid d_j$, $1\leq i\neq j\leq k$.
Clearly, the class of graphs $\ICG_{2m}(D_0)$ found in Theorem 5 in \cite{severini} has the smallest parameters $n$ and $d_1, \ldots, d_k$ that satisfy the  condition above.

\smallskip

According to Remark \ref{rem:1} we can derive new classes of graphs whose maximal diameters attain the upper bounds $r(n)$ or $r(n)+1$ in some special cases

\begin{theorem}
For a given positive integer $n>1$, the diameter of
$\ICG_n(D)$ is equal to $r(n)$ in the following cases
\begin{itemize}
\item[i)] $n=p_1p_2\in 2\N+1$, $D=\{1\}$,
\item[ii)] $n=4p_2$, $D=\{1\}$,
\item[iii)] $n=2p_2p_3$, $D\in\{\{1\},\{1,p_2\},\{2,p_2\},\{p_2,p_3\},\{1,p_2,p_3\}\}$,
\item[iv)] $n=2p_2^2$, $D\in\{\{1\},\{1,p_2\}\}$.
\end{itemize}
The diameter of $\ICG_n(D)$ is equal to $r(n)+1$ if
$n=2p_2$, $D=\{1\}$.
\end{theorem}
\begin{proof}
Notice that $diam(\ICG_{n\cdot n'}(D))=3$ if $diam(\ICG_n(D))=1$ and $n'\in 2\N$. We see that $diam(\ICG_{n\cdot n'}(D))$ attains $r(n\cdot n')$ if
$n\cdot n'\in\{2p_2^2,4p_2,2p_2p_3\}$. Since $\gcd(n,n')=1$, we further have $(n',n)\in\{(2,p_2^2), (4,p_2), (2,p_2p_3), (2p_2,p_3)\}$. If $n=p_2$ then $diam(\ICG_n(D))=1$ for $D=\{1\}$, if $n=p_2^2$ then $diam(\ICG_n(D))=1$ for $D=\{1,p_2\}$ and if $n=p_2p_3$ then $diam(\ICG_n(D))=1$ for $D=\{1,p_2,p_3\}$.
Furthermore, $diam(\ICG_{n\cdot n'}(D))$ attains $r(n\cdot n')+1$ if
$n\cdot n'=2p_2$ and therefore $(n',n)=(2,p_2)$, $D=\{1\}$.

The diameter of $\ICG_{n\cdot n'}(D)$ is equal to $2$ if $diam(\ICG_n(D))=1$ and $n'\in 2\N+1$. Moreover, $diam(\ICG_{n\cdot n'}(D))$ attains $r(n\cdot n')$ if $(n,n')=(p_1,p_2)$ and $diam(\ICG_n(D))=1$ if $D=\{1\}$.

The diameter of $\ICG_{n\cdot n'}(D)$ is equal to $3$ if $diam(\ICG_n(D))=2$ and $n\in 2\N+1$. Moreover, $diam(\ICG_{n\cdot n'}(D))=3$ attains $r(n\cdot n')$ if $(n',n)\in\{(2,p_2p_3), (2p_2,p_3),(2,p_2^2)\}$. If $n=p_2p_3$ then $diam(\ICG_n(D))=2$ for $D\in\{\{1\},\{1,p_2\},\{p_2,p_3\}\}$, and  if $n=p_2^2$ then $diam(\ICG_n(D))=2$ for $D=\{1\}$.

The diameter of $\ICG_{n\cdot n'}(D)$ is equal to $3$ if $diam(\ICG_n(D))=2$, $n\in 2\N$ and there exists $0\leq l\leq n-1$ such that $(0,l)$ are adjacent and there is no $0\leq l_1\neq l\leq n-1$ such that $(0,l_1)$ are adjacent and $(l_1,l)$ are adjacent.
Moreover, $diam(\ICG_{n\cdot n'}(D))=3$ attains $r(n\cdot n')$ if $(n',n)\in\{(p_2p_3,2), (p_2,2p_3),(p_2^2,2), (p_2,4)\}$. If $n=2p_3$ then $D\in\{\{1,p_3\},\{2,p_3\}\}$ and if $n=4$ then $D=\{1\}$.

\end{proof}

\bigskip

In the following assertion without loss of generality we assume that the indices $\{i_1,\ldots,i_t\}$ from the properties
(\ref{property divisors}) and (\ref{property divisors1}) are equal $\{1,\ldots,t\}$.
Furthermore, we also suppose that for every prime $p_i$, $1\leq i\leq k$, there exists at least one $d_j$, $1\leq j\leq t$, such that $p_i\mid d_j$ (in the opposite case we have already calculated the diameter  in Theorem \ref{thm:main}).

\begin{theorem}
\label{thm:t<k}
Let $ \ICG_n (d_1,
d_2,\ldots, d_t)$ be an integral circulant graph such that $t < k$ and the set of the divisors $D=\{d_1, d_2,\ldots,
d_t\}$ satisfies the properties (\ref{property divisors}) and (\ref{property divisors1}) with respect to primes $\{p_{1},p_{2},\ldots,p_{t}\}$.
Then, the diameter of the of the graph is less or equal to
\begin{itemize}
\item[i)] $r(n)$, if $n\not\in 4\N+2$ and there exist at least two divisors $p_i$ and $p_j$, $1\leq i\neq j\leq k$, such that $\alpha_i=\alpha_j=1$. The equality holds if and only if $\{d_1,\ldots, d_t\}$ satisfies ($i$) from Theorem \ref{thm:r(n)}, there exist $p_1,\ldots,p_t\in 2\N+1$ and an injection $f:\{{t+1},\ldots,k\}\rightarrow\{1,\ldots,t\}$ such that for every prime
    $p_j$, $t+1\leq j\leq k$, there exists exactly one $d_{f(j)}$ such that $p_j\nmid d_{f(j)}$ and $\alpha_j=\alpha_{f(j)}=1$.

\item[ii)] $r(n)+1$, if $n\in 4\N+2$ and there exists at least one divisor $p_i\in 2\N+1$, $1\leq i\leq k$, such that $\alpha_i=1$. The equality holds if and only if $\{d_1,\ldots, d_t\}$ satisfies ($i$) from Theorem \ref{thm:r(n)}, there exist $p_1,\ldots,p_t\in 2\N+1$ and an injection $f:\{{t+1},\ldots,k\}\rightarrow\{1,\ldots,t\}$ such that for every prime
    $p_j$, $t+1\leq j\leq k$, there exists exactly one $d_{f(j)}$ such that $p_j\nmid d_{f(j)}$ and $\alpha_j=\alpha_{f(j)}=1$.
\item[iii)] $2t$, if $n\in 2\N+1$ and there exists at most one divisor $p_i$, $1\leq i\leq k$, such that $\alpha_i=1$. The equality holds if and only if there exist $p_1,\ldots,p_t$ such that for every $d_j$, $1\leq j\leq t$, $p_{i}^2\mid d_j$, $1\leq i\neq j\leq t$.
    
\item[iv)] $2t+1$, if $n\in 2\N$ and there exists at most one divisor $p_i$, $1\leq i\leq k$, such that $\alpha_i=1$. The equality holds if and only if there exist $p_1,\ldots,p_t\in 2\N+1$ such that for every $d_j$, $1\leq j\leq t$, $p_{i}^2\mid d_j$, $1\leq i\neq j\leq t$.
\end{itemize}
\end{theorem}

\begin{proof}
First we prove the parts $i)$ and $ii)$ of the assertion.
According to the proof of Lemma \ref{lem:auxiliary} we show that for every prime factor $p_j$, $1\leq j\leq t$, we need at most one edge if $\alpha_j=1$ and two edges if $\alpha_j>1$ in order to construct a path from $0$ to $l$, $0\leq l\leq n-1$. Therefore, we can write
\begin{eqnarray}
\label{eqn:modulo pjalphaj}
d_1 y_1 + d_2 y_2 +\cdots + d_t y_t \equiv l \pmod {p_{j}^{\alpha_{j}}},
\end{eqnarray}
for $1\leq j\leq t$, where some $y_i$, $1\leq i\leq t$, need to be split into two summands $y_i'+y_i''$ for $\alpha_i>1$ and certain values of $l$ modulo $p_{j}^{\alpha_{j}}$.
Notice also that some of $y_i$ can be equal to zero.

In the following, we examine how many edges we can add at the most to obtain the diameter path when we consider modulus $p_j$, $t+1\leq j\leq k$.
Now, let $p_j$ be a prime factor of $n$ such that $t+1\leq j\leq k$. Observe that the relation (\ref{eqn:modulo pjalphaj}) can be
rewritten in the following form
\begin{eqnarray}
\label{auxiliary form}
s_1+\cdots+s_{u_j}+p_j(s_{u_j+1}+\cdots+s_{t}) \equiv l \pmod {p_{j}^{\alpha_{j}}},
\end{eqnarray}
where $u_j=|\{d_i\ |\ i\in\{1,\ldots,t\},\
 \ p_j\nmid d_i\}|$, $s_i=d_i y_i$, for $1\leq i\leq u_j$ and $p_is_i=d_i y_i$, for $u_j+1\leq i\leq t$.

Suppose $p_j>2$.
Using a  consideration similar to the one given in the proof of Theorem \ref{thm:r(n)} for the existence of the solution $t_1,\ldots,t_{k-1}$ in the equation
(\ref{eqn: mod l})
 we can conclude in the same way that there exist $s_1,\ldots,s_t$ modulo $p_j$ for $u_j>1$, such that $p_j\nmid y_1,\ldots,y_t$. This means that we do not need any additional edges for the path from $0$ to $l$
 corresponding to the modulus $p_j$.
However, if $u_j=1$ then it must be that $p_j\mid s_1$ for $l$ divisible by $p_j$, which is a contradiction.
Furthermore, as $p_j\in 2\N+1$  and $p_j\nmid d_1$, according to Lemma \ref{lem:2 and 3 summands} it holds that $y_1$ can be split into two summands $y_1'+y_1''$,
since there exist $y_1'$ and $y_1''$ not divisible by $p_j$ such that $y_1'+y_1''\equiv (l-p_jt)d_1^{-1} \pmod {p_j^{\alpha_j}}$, where $p_jt=d_{u_j+1}y_{u_j+1}+\cdots+d_ty_t$.
Therefore, in this case we need two edges for the parts of the path corresponding to the moduli $p_1$ and $p_j$.
Finally, we conclude that the path length is increased by one if one edge for the part of the path corresponding to the modulus $p_1$ is needed and that is for $\alpha_1=1$.
Also notice that these two edges can reach the value $r(p_1^{\alpha_1})+r(p_j^{\alpha_j})$ if and only if $\alpha_j=1$.
If there are two primes $p_{j_1}>2$ and $p_{j_2}>2$ not dividing $d_1$, for $t+1\leq j_1,j_2\leq k$, it can be immediately seen that there exist $y_1',y_1''$ modulo $p_{j_i}$ such that
$p_{j_i}\nmid y_1',y_1''$, $1\leq i\leq 2$ and path length can not be increased.
 Thus we prove that the path length is increased by one if and only if an arbitrary divisor is not divisible by exactly two prime factors $p_i$ and $p_j$, where $1\leq i\leq t$, $t+1\leq j\leq k$ and $\alpha_i=\alpha_j=1$. Now, if we take $l$ such that the equation (\ref{eqn:modulo pjalphaj}) has maximal number of summands and that is $r(p_1^{\alpha_1}\cdots p_t^{\alpha_t})$, we obtain that the diameter of the graph is equal to $r(n)$ for $n\in 2\N+1$.
In this way, we also prove that the diameter is smaller than $r(n)$ and $r(n)+1$, for $n\in 4\N$ and $4\N+2$, respectively if $p_1=2$.
Therefore, in the following we assume that $p_1,\ldots,p_t\in 2\N+1$.

\smallskip


Suppose now that $n\in 2\N$. If $p_j=2$, similarly to the  discussion above we can conclude that there exist $s_1,\ldots,s_t$ modulo $p_j$ for $u_j>2$ and we do not need any additional edges for the path from $0$ to $l$. Furthermore, if $u_j=1$, according to Lemma \ref{lem:2 and 3 summands} it holds that $y_1$ can be split into three summands $y_1'+y_1''+y_1'''$, $p_j\nmid y_1',y_1'',y_1'''$.
In this case we need three edges for the parts of the path corresponding to the moduli $p_1$ and $p_j$ and these three edges reach the value
$r(p_1^{\alpha_1})+r(p_j^{\alpha_j})+1$ if and only if $\alpha_1=\alpha_j=1$. This proves that under the conditions given in the part $ii)$ of the assertion the diameter of the graph attains $r(n)+1$. Moreover, if $n\in 4\N$ then $r(p_1^{\alpha_1})+r(p_j^{\alpha_j})+1>3$ and $r(p_1^{\alpha_1})+r(p_j^{\alpha_j})=3$ if and only if $\alpha_1=1$. Thus
we conclude that under the conditions from the part $i)$ the diameter of the graph attains $r(n)$, for $n\in 4\N$.

To prove the assertions $i)$ and $ii)$ completely, we need to prove that if $p_j=2$ and $u_j=2$ the path from $0$ to $l$ can not be greater or equal $r(n)+1$ for $n\in 4\N+2$ and
$r(n)$ for $n\in 4\N$. Indeed, according to Lemma \ref{lem:2 and 3 summands} again, we need three edges ($y_1',y_1''$ and $y_2$) for the parts of the path corresponding to the moduli $p_1$, $p_2$ and $p_j$ and since $r(p_1^{\alpha_1})+r(p_2^{\alpha_2})+r(p_j^{\alpha_j})$ is always greater or equal to three, it can be seen that
$r(p_1^{\alpha_1})+r(p_2^{\alpha_2})+r(p_j^{\alpha_j})+1> 3$ for $n\in 4\N+2$ and $r(p_1^{\alpha_1})+r(p_2^{\alpha_2})+r(p_j^{\alpha_j})> 3$ for $n\in 4\N$.

\bigskip

Now, we are going to prove the parts $iii)$ and $iv)$ of the assertion. First, suppose that there exist $p_1,\ldots,p_t$ such that $\alpha_1,\ldots,\alpha_t>1$.
According to the proof of Lemma \ref{lem:auxiliary}  the path from $0$ to $l$, for certain values $0\leq l\leq n-1$ and $d_1,\ldots,d_t$ that satisfy $i)$ from Theorem \ref{thm:r(n)},
 can be written in the form
\begin{eqnarray*}
d_1 (y_1'+y_1'') + d_2 (y_2'+y_2'') + \cdots + d_t (y_t'+y_t'') \equiv l \pmod {p_{j}^{\alpha_{j}}},
\end{eqnarray*}
for $1\leq j\leq t$. Following the above deduction we conclude that for any $p_j>2$, $t+1\leq j\leq k$ such that $p_j\nmid d_s$, $1\leq s\leq t$, we can find $y_s'$ and $y_s''$ modulo $p_j$ such that $p_j\nmid y_s',y_s''$. Similarly, if $p_j=2$ and $p_j\nmid d_s$ for some $1\leq s\leq t$, then we can not find $y_s'$ and $y_s''$ modulo $p_j$ such that $p_j\nmid y_s',y_s''$. Therefore, we replace those two edges $y_s'+y_s''$ with three edges $y_s'+y_s''+y_s'''$ modulo $p_j$ such that $p_j\nmid y_s',y_s'',y_s'''$.
So, if $n\in 2\N+1$ we can not find $l$ such that we need more than $2t$ summands and for $n\in 2\N$ it is possible by increasing the number of the summands by one.

Finally, if there are no $p_1,\ldots,p_t$ such that $\alpha_1,\ldots,\alpha_t>1$ then there exists a divisor $d_i$ such that $p_u\nmid d_i \Rightarrow\alpha_u=1$ for $1\leq u\leq k$. On the other hand, since there exists at most one prime factor $p_i$ of $n$ such that $\alpha_i=1$, we see that $p_i$ is the only factor dividing $d_i$ (if it exists) and
for all other prime factors it holds that $p_u\mid d_i$, $1\leq u\neq i\leq k$, where $\alpha_u>1$.
According to the proof of Lemma \ref{lem:auxiliary} the path from $0$ to $l$, for certain values $0\leq l\leq n-1$ can be written in the form
\begin{eqnarray*}
\label{eqn:modulo pj^alphaj}
d_1 (y_1'+y_1'') +\cdots+ d_i y_i +\cdots + d_t (y_t'+y_t'') \equiv l \pmod {p_{j}^{\alpha_{j}}},
\end{eqnarray*}
for $1\leq j\leq t$ (the maximal number of summands is equal to $2t-1$ considering to the moduli $p_1,\ldots,p_t$). As we have seen above for any $p_j>2$, $t+1\leq j\leq k$ such that $p_j\nmid d_s$, $1\leq s\neq i\leq t$ we can not increase the number of the summands and for
$p_j=2$, $t+1\leq j\leq k$ such that $p_j\nmid d_s$, $1\leq s\neq i\leq t$ we can increase the number of the summands by one, but we can not reach the value $2t+1$.
 \end{proof}

\begin{remark}
According to $i)$ and $ii)$ of the assertion we may notice that there is a subset $\{d_{f(t+1)},\ldots,d_{f(k)}\}$ of the divisor set $D$ such that there exist exactly two divisors $p_i,p_{f(i)}\nmid d_{f(i)}$ and $p_i,p_{f(i)}\| n$, for $t+1\leq i\leq k$.
Now, if $s(n)=|\{ i \ | \alpha_i= 1,\ 1\leq i\leq k \}|$ we conclude that $2(k-t)\leq s(n)$ and therefore $k-\lfloor{\frac{s(n)}{2}}\rfloor \leq t<k$. So, the assertions $i)$ and $ii)$ hold for the divisor set with cardinality  $k-\lfloor{\frac{s(n)}{2}}\rfloor \leq t<k$ and for the cardinalities $t<k-\lfloor{\frac{s(n)}{2}}\rfloor$ it can be very easily seen that $diam(\ICG_n(d_1,\ldots,d_t))\leq 2t$ if $n\in 2\N+1$ and $diam(\ICG_n(d_1,\ldots,d_t))\leq 2t+1$ if $n\in 2\N$ (the reasoning is analogous to that in the parts $iii)$ and $iv)$).
\end{remark}

\bigskip

Notice that we can neither increase nor find new classes of graphs whose diameters attain $r(n)$ or $r(n)+1$ using Theorems \ref{thm:d,d+1} and \ref{thm:t<k}, where $n$ is the order of the integral circulant graph.
Indeed, if $n'\in 2\N+1$ then the diameter of $\ICG_{n\cdot n'}(d_1,\ldots,d_t)$ remains the same and the potential upper bound $r(n\cdot n')$ can not be attained as $r(n\cdot n')>r(n)$.
Furthermore, if $n'\in 2\N$ then $n\in 2\N+1$ and
$$
diam(\ICG_{n\cdot n'}(d_1,\ldots,d_t))= diam(\ICG_n(d_1,\ldots, d_t))+1\leq r(n)+1\leq r(n\cdot n').
$$
It is easy to see that if $diam(\ICG_{n\cdot n'}(d_1,\ldots,d_t))= r(n\cdot n')$  then $n'=2$ and therefore $n\cdot n'\in 4\N+2$.
However, since $n\cdot n'\in 4\N+2$, according to the part $ii)$ of Theorem \ref{thm:t<k} there exist $d_1',\ldots,d_t'$ such that  $diam(\ICG_{n\cdot n'}(d_1,\ldots, d_t))$ is equal to $r(n\cdot n')+1$.

\smallskip
Now, if $n\cdot n'$ is divisible  by at most one prime factor then new classes of graphs that attain the bounds $2t$ and $2t+1$ can be found in the cases
$n\cdot n'\in 2\N+1$ and $n\cdot n'\in 2\N$, respectively.

On the other hand, if $n$ has at most one divisor $p_i\|n$ and $n\cdot n'$ is divisible by at least two divisors $p_i$ and $p_j$ such that $p_i,p_j\|n\cdot n'$ and  then we have for $n\in 2\N+1$, $n'\in 2\N$ and $D=\{d_1,\ldots,d_t\}$
$$
diam(\ICG_{n\cdot n'}(D))= diam(\ICG_n(D))+1\leq 2t+1\leq 2(k-1)+1\leq r(n)< r(n\cdot n').
$$
Similarly, if $n'\in 2\N+1$ and $n\in 2\N$ we can similarly conclude that \\ $diam(\ICG_{n\cdot n'}(d_1,\ldots,d_t))= diam(\ICG_n(d_1,\ldots, d_t))\leq 2t+1< r(n\cdot n').$
Since  $n\cdot n'$ is divisible by at least two divisors $p_i$ and $p_j$ such that $\alpha_i=\alpha_j=1$, according to $i)$ and $ii)$ of the above assertion there exists a set of the divisors
$d_1',\ldots,d_t'$ such that $diam(\ICG_{n\cdot n'}(d_1',\ldots,d_t'))\geq r(n\cdot n')$ and thus we can neither increase nor find new classes of graphs whose diameters attain maximal values using Theorem \ref{thm:d,d+1}.

Finally, we give the formula for the maximal diameter of an integral circulant graph of a given order $n$ and the set of divisors $D$ with cardinality  $t$
\begin{eqnarray}
\label{eqn:main}
\max(diam(\ICG_{n}(D)))=\left\{ \begin{array}{rl}
r(n),&  t=k \\
r(n)+1,& n\in 4\N+2,\ s(n)\geq 2,\ k-\lfloor{\frac{s(n)}{2}}\rfloor \leq t<k\\
r(n),& n\not\in 4\N+2,\ s(n)\geq 2,\ k-\lfloor{\frac{s(n)}{2}}\rfloor \leq t<k\\
2t+1,& n\in 2\N,\ s(n)\geq 2,\ t<k-\lfloor{\frac{s(n)}{2}}\rfloor \\
2t,& n\in 2\N+1,\ s(n)\geq 2,\ t<k-\lfloor{\frac{s(n)}{2}}\rfloor \\
2t+1,& n\in 2\N,\ s(n)\leq 1,\ t<k \\
2t,& n\in 2\N+1,\ s(n)\leq 1,\ t<k. \\
\end{array}\right.
\end{eqnarray}

In some real-world applications such as construction of a quantum communication network with prescribe order and the largest possible diameter it is good to know
the maximal diameter of integral circulant graphs of a given order $n$.

It is sufficient to check if there exists some $\ICG_n(D)$ whose
maximal diameter is $2|D|+1$ or $2|D|$ and it is greater than $r(n)+1$ or $r(n)$.
From  (\ref{eqn:main}) it follows that diameter of $\ICG_n(D)$ can attain $2|D|+1$ or $2|D|$, if $s(n)\leq 1$ or $s(n)\geq 2$ and
$t<k-\lfloor{\frac{s(n)}{2}}\rfloor$. Now, suppose first that $s(n)\leq 1$. The following chain of inequalities holds
$$
diam(\ICG_n(D))\leq 2|D|+1\leq 2(k-1)+1\leq 2k-s(n)= s(n)+2(k-s(n))=r(n),
$$
and therefore we see that in this case can not find a class of graphs with  diameter greater than $r(n)$.

Now, if $s(n)\geq 2$ and $t<k-\lfloor{\frac{s(n)}{2}}\rfloor$ then we have
$$
diam(\ICG_n(D))\leq 2|D|+1\leq 2(k-\lfloor{\frac{s(n)}{2}}\rfloor-1)+1=2k-(s(n)-1)-1 =r(n),
$$
and therefore we conclude that the maximal diameter is equal to $r(n)$ if $n\not\in 4\N+2$ and $r(n)+1$ if $n\in 4\N+2$.

\section{Conclusion}

In this paper we find the maximal diameter of all integral circulant
graphs of a given order $n$. We also calculate the maximal diameter of integral circulant
graphs of a given order $n$ and cardinality of the divisor set $t\leq k$.
Moreover, we characterize all
$\ICG_n(D)$ such that the maximal diameter is attained.
 Generally, the proofs presented in this paper are based
on the connections between  number and  graph theory
and fall into a good many of distinct cases.

We have already mentioned that the maximal diameter of some class of graphs of a given order plays an important role
if that class is used for modeling a quantum network that allows quantum dynamics, having in mind applications like perfect state transfer.
Therefore, we conclude with a partial result about maximal diameter of integral circulant graphs having perfect state transfer.
We actually prove that diameter of $\ICG_n(D)$ having perfect state transfer such that $|D|\leq k$ can not attain the value $r(n)$.
In a recent paper \cite{Ba10}, complete characterization
of integral circulant graph having perfect state transfer was given
and it was shown that $\ICG_n (D)$ has perfect state transfer if and only if $n\in 4\N$
and $D=\widetilde{D_3}\cup D_2\cup 2D_2\cup 4D_2\cup \{n/2^a\}$,
where $\widetilde{D_3}=\{d\in D\ |\ n/d\in 8\N\}$, $D_2= \{d\in D\
|\ n/d\in 8\N+4\}\setminus \{n/4\}$ and $a\in\{1,2\}$. Now, suppose that there exists $\ICG_n(D)$ of a given order $n$ allowing
perfect state transfer that attains maximal diameter. Since $\ICG_n(D)$ attains maximal diameter we conclude that $D$ satisfies
(\ref{property divisors}) and (\ref{property divisors1}) and therefore there exists $p_i$ not dividing $n/2$, if $n/2\in D$.
This means that $p_i=2$ and hence $n\in 4\N+2$ implying that $\ICG_n(D)$ does not permit perfect state transfer.
Furthermore, if $n/4\in D$ then in the similar way we conclude that $n\in 8\N+4$. On the other hand, for every $d_i\in D_2$ it holds that $2d_i,4d_i\in D$ and the set of prime factors dividing  $2d_i$ and $4d_i$ are the same, which is a contradiction according to (\ref{property divisors}) and (\ref{property divisors1}).
The above discussion suggests that a possible challenging direction in future research would be finding the maximal diameter in the class of integral circulant graphs having perfect state transfer.

On the other hand, we may notice that we can not improve the lower bound in the inequality (\ref{enq:main}) for prescribed $n$ and any prescribed  cardinality of the divisor set $D$.
Indeed, according to Theorem 9 from \cite{klotz07} we observe that $\ICG_n(1)=2$ if and only if $n$ is a power of $2$ or $n$ is odd (in both of the cases $n$ is not prime). This implies that $diam(\ICG_n (1,d_2,\ldots,d_t))=2$, for any $t$ and the mentioned values of $n$ (as long as $\{1,d_2,\ldots,d_t\}\neq D_n$ ).

In the remaining case, for $n=2^{\alpha_1}m$, where $m>1$ is odd and $\alpha\geq 1$, we can prove that $diam(\ICG_n(1,2^{\alpha_1}))=2$. Indeed, for every $0\leq l\leq n-1$, such that $l$ is even we will
prove the existence of $l$ in the  form $s_1+s_2\equiv l \pmod{n}$ such that $\gcd(s_1,n)=1$, $\gcd(s_2,n)=1$.
Now, let $p_i$ be an arbitrary divisor of $n$ such that
$S_{p_i}(n)=\alpha_i$.
By solving the above congruence
equation system modulo $p_i^{\alpha_i}$ we get that $s_1\equiv
l-s_2\not\equiv p_iu\pmod {p_{i}^{\alpha_i}}$ and
$s_2\not\equiv 0\pmod {p_{i}}$ (which is equivalent to
$s_2\not\equiv p_iv\pmod {p_{i}^{\alpha_i}}$), for $0\leq
u,v<p_i^{\alpha_i-1}-1$. Therefore, we obtain $s_2\not\equiv
\{p_iv,l-p_iu\}\pmod {p_{i}^{\alpha_i}}$ and it can be
concluded that  the maximal number of values that $s_2$ modulo $p_i^{\alpha_i}$
can not take is equal to $2p_i^{\alpha_i-1}$.  This
number is  less than the number of residues modulo $p_i^{\alpha_i}$
any $p_i$ (as $l$ is even), so this
system has a solution in this case.

Now, suppose that $l\in 2\N+1$. We find
$s_1$ and $s_2$ such that $s_1+s_2\equiv l \pmod{n}$, $\gcd(s_1,n)=1$ and $\gcd(s_2,n)=2^{\alpha_1}$.
Similarly to the previous discussion, we can conclude that the above conditions can be reduced to the following system
$s_2\not\equiv
\{p_iv,l-p_iu\}\pmod {p_{i}^{\alpha_i}}$, for $2\leq i\leq k$ and $s_2\equiv
0\pmod {2^{\alpha_1}}$. As $p_i>2$ this system has a solution. We finally conclude that
$diam(\ICG_n (1,2^{\alpha_1},d_3,\ldots,d_t))=2$, for any $t$ and the mentioned values of $n$ (as long as $\{1,2^{\alpha_1},d_3,\ldots,d_t\}\neq D_n$).

\section*{Declaration of competing interest}

The authors declare that they have no known competing financial
interests or personal relationships that could have appeared to
influence the work reported in this paper.

\section*{Data availability}

No data was used for the research described in the article.

\section*{Acknowledgments}

Authors gratefully acknowledge support from the Research Project
of the Ministry of Education, Science and Technological Development of the Republic of Serbia (number 451-03-47/2023-01/ 200124).












\begin{thebibliography}{99}

\bibitem{Ba10}
    M. Ba\v si\'c, Characterization of circulant graphs having perfect state transfer,
     Quantum Information Processing 12 (2013) 345--364.

\bibitem{Ba14} M. Ba\v si\'c, Which weighted circulant networks have perfect state transfer?,
Information Sciences 257 (2014) 193--209.


\bibitem{basic08}
M. Ba\v si\'c, A. Ili\'c, On the clique number of integral
circulant graphs, Applied Mathematics Letters 22 (2009)
1406--1411.

\bibitem{bail11}
M. Ba\v si\' c, A. Ili\' c, On the automorphism group of integral circulant graphs, The Electronic Journal of Combinatorics 18 (2011) \#P68.

\bibitem{berrizbeitia04}
    P. Berrizbeitia, R. E. Giudic, On cycles in the sequence of unitary Cayley graphs,
    Discrete Mathematics 282 (2004) 239--243.

\bibitem{fizicarski}
M. Christandl, N. Datta, A. Ekert, A.J. Landahl, Perfect
state transfer in quantum spin networks, Physical Review Letters 92
(2004) 187902 [quant-ph/0309131].


\bibitem{fuchs05}
E. Fuchs, Longest induced cycles in circulant graphs,
     The Electronic Journal of Combinatorics 12 (2005) 1--12.


\bibitem{HardyWright}
G.H. Hardy, E.M. Wright, An introduction to the Theory of
Numbers, 5th ed, Clarendon Press, Oxford University Press, New
York, 1979.

\bibitem{hwang03}
    F. K. Hwang, A survey on multi-loop networks, Theoretical Computer Science 299 (2003) 107--121.

\bibitem{Ilic09}
A. Ili\'c, Distance spectra and distance energy of integral
circulant graphs, Linear Algebra and its Applications 433 (2010)
1005--1014.


\bibitem{Ilic10}
A. Ili\'c, M. Ba\v si\'c, On the chromatic number of
integral circulant graphs, Computers \& Mathematics with
Applications 60 (2010)  144--150.

\bibitem{IlBa11}
A. Ili\'c, M. Ba\v si\'c, New results on the energy of integral circulant graphs, Applied Mathematics and Computation 218 (2011) 3470--3482.


\bibitem{klotz07}
    W. Klotz, T. Sander, Some properties of unitary Cayley graphs,
  The Electronic Journal of Combinatorics 14 (2007) \#R45.

\bibitem{PaCh00}
J.H. Park, K.Y. Chwa, Recursive circulants and their
embeddings among hypercubes, Theoretical Computer Science 244 (2000) 35--62.


\bibitem{RaVe09}
    H. N. Ramaswamy, C. R. Veena,
    On the Energy of Unitary Cayley Graphs,
   The Electronic Journal of Combinatorics 16 (2007) \#N24.


\bibitem{severini}
N. Saxena, S. Severini, I. Shparlinski, Parameters of
integral circulant graphs and periodic quantum dynamics,
    International Journal of Quantum Information 5 (2007) 417--430.

\bibitem{SaSa12}
J.W. Sander, T. Sander, The maximal energy of classes of integral circulant graphs, Discrete Applied Mathematics 160 (2012) 2015--2029.

\bibitem{wasin}
W. So, Integral circulant graphs, Discrete Mathematics 306
(2006) 153--158.


\end{thebibliography}







\end{document}